\documentclass[journal,twoside,web]{ieeecolor}
\usepackage{generic}
\usepackage{cite}
\usepackage{amsmath,amssymb,amsfonts}
\usepackage{algorithmic}
\usepackage{graphicx}
\usepackage{algorithm,algorithmic}
\usepackage{hyperref}
\hypersetup{hidelinks=true}
\usepackage{textcomp}

\let \dis \displaystyle

\let \bb \mathbb
\let \rm \mathrm
\let \cal \mathcal 

\let \op \operatorname
\let \eps \varepsilon

\newtheorem{thm}{Theorem}
\numberwithin{thm}{section}
\newtheorem{prop}[thm]{Proposition}

\newtheorem{lem}[thm]{Lemma}

\newtheorem{rem}[thm]{Remark}

\def\BibTeX{{\rm B\kern-.05em{\sc i\kern-.025em b}\kern-.08em
    T\kern-.1667em\lower.7ex\hbox{E}\kern-.125emX}}
\begin{document}
\title{Tracking controllability for the heat equation}
\author{Jon Asier B\'arcena-Petisco,  Enrique Zuazua
\thanks{We would like to thank Professor Sebasti\'an Zamorano and the anonymous referees for their useful comments and valuable feedback. The work of J.A.B.P.
is funded by the Basque Government, under grant~IT1615-22; and 
by the project PID2021-126813NB-I00 funded by MICIU/AEI/10.13039/501100011033 and by ``ERDF A way of making Europe''.
The work of E.Z.  has been funded by the Alexander von Humboldt-Professorship program, the ModConFlex Marie Curie Action, HORIZON-MSCA-2021-DN-01, the COST Action MAT-DYN-NET, the Transregio 154 Project Mathematical Modelling, Simulation and Optimization Using the Example of Gas Networks of the DFG, AFOSR  24IOE027 project, grants PID2020-112617GB-C22, TED2021-131390B-I00 of MINECO and PID2023-146872OB-I00 of MICIU (Spain),
Madrid Government - UAM Agreement for the Excellence of the University Research Staff in the context of the V PRICIT (Regional Programme of Research and Technological Innovation). 
 }
\thanks{[1] Department of Mathematics, University of the Basque Country, 48080, Bilbao Spain (e-mail: jonasier.barcena@ehu.eus). }
\thanks{[2]  Chair for Dynamics, Control, Machine Learning, and Numerics (Alexander von Humboldt-Professorship), Department of Mathematics, Friedrich-Alexander-Universit\"at Erlangen-N\"urnberg, 91058 Erlangen, Germany. 
[3] Chair of Computational Mathematics, Fundaci\'on Deusto, Avenida de las Universidades 24, 48007 Bilbao, Basque Country, Spain.
[4] Departamento de Matem\'aticas, Universidad Aut\'onoma de Madrid, 28049 Madrid, Spain. (e-mail: enrique.zuazua@fau.de).}}

\maketitle

\begin{abstract}
We study the tracking or sidewise controllability of the heat equation. More precisely, we seek for controls that, acting on  part of the boundary of the domain where the heat process evolves, aim to assure that the normal trace or flux on the complementary set tracks a given trajectory. 

The dual  equivalent observability problem is identified. It consists on estimating the boundary sources, localized on a given subset of the boundary, out of boundary measurements on the complementary subset. 

Classical unique continuation and smoothing properties of the heat equation allow us proving  approximate tracking controllability properties and the smoothness of the class of trackable trajectories. 

We also develop a new transmutation method which allows to transfer known results on the sidewise controllability of the wave equation to the tracking controllability of the heat one. 

Using the flatness approach we also give explicit estimates on the cost of approximate tracking control. 

The analysis is complemented with a discussion of some possible variants of these results and a list of open problems. 
\end{abstract}


\section{Introduction}

In this paper we analyze the  tracking or sidewise controllability problem for the heat equation:
\begin{equation}\label{con:heat}
\begin{cases}
y_{t}-\Delta y=0 & \mbox{ in } (0,T)\times\Omega,\\
y=v1_\gamma & \mbox{ on } (0,T)\times\partial\Omega,\\
y(0)=y_0 & \mbox{ in }\Omega,
\end{cases}
\end{equation}
when $\Omega\subset\bb R^d$ is a given open bounded domain, $T>0$ is a given time horizon,
  $\gamma\subset\partial\Omega$ is a subset of the boundary,
  $v$ is the control, and $y_0$ is the initial value. 
Hereafter, we denote by $1_\gamma$ the characteristic function
of the set $\gamma$ of the boundary where the source term acts.

The \emph{sidewise or tracking controllability problem} is as follows: given
$\tilde \gamma\subset \partial\Omega$ (usually, but not necessarily,
$\tilde \gamma \subset\partial\Omega\setminus\gamma$),
 and a sufficiently regular target $w$, to find a control $v$ in an appropriate space such that:
\begin{equation}\label{eq:parny}
\partial_\nu y=w \ \ \ \ \mbox{ on } (0,T)\times\tilde\gamma,
\end{equation}
where $\nu$ denotes
the normal vector to $\partial\Omega$ pointing outwards. 
In other words, we seek to control the flux  on $(0,T)\times\tilde \gamma$ by acting on $(0,T)\times \gamma$. When such a control $v$ exists, so that \eqref{eq:parny}
is satisfied, the target $w$ is said to be \emph{reachable} or \emph{trackable}. 

Of course, analogous problems can be considered, with the same techniques, for other boundary conditions on the control and the target trace. For instance, we can replace the Dirichlet control, $y=v1_\gamma$, by the Neumann one, $\partial_{\nu}y=v1_\gamma$, and the target
$\partial_\nu y=w$ by $y=w$. 

The potential applications of this and similar control problems include  the goal-oriented and localized control of the temperature or its 
flux (see, for example, 
\cite{pereira2009lethal},
\cite{jerby2017localized}
and \cite{hannon2021effects}). These problems are relevant also in the context of population dynamics where the regulation of the flux of population across borders is often a sensitive and relevant issue, \cite{ruizzuazua}.

These problems are also relevant and can be formulated for other models, such as the wave equation.  Actually, we will establish a correlation between the tracking controllability of the heat and wave equations through a new   subordination or transmutation principle.

In the particular 1$d$ case, the reachable space for the heat equation has been analyzed in the pioneering work  
 \cite{laroche2000motion},  
by using power series representation methods in the context of motion planning. 
Other works on 1$d$ parabolic equations in which the control of boundary traces is discussed 
include \cite{dunbar2003motion}, \cite{lynch2002flatness}, \cite{martin2016null}, \cite{martin2016reachable} and \cite{schorkhuber2013flatness}. 
 In the multi-dimensional setting 
 the  known results are only valid for cylinders
(see \cite{martin2014null}), where separation of variables can be employed, reducing the problem to the
1$d$ case.

In the present paper, first, in Section \ref{sec:gendom},
by duality, we transform the tracking controllability
 problem on its dual observability one, which consists on identifying
heat  sources on part of the boundary of the domain  out of 
measurements done on another observation subdomain. 
This observability problem differs from classical ones on the fact that, normally,  the initial data of the system is the object to be  identified.
 
 Duality, together with the Holmgren's Uniqueness Theorem,
 allows  to prove easily the approximate tracking controllability property, i.e. the fact that \eqref{eq:parny} can be achieved for all target up to an arbitrarily small $\varepsilon$ error.

Second, in Section \ref{sec:trasm}, using a new transmutation formula, inspired on the classical Kannai transform \cite{kannai1977off}, we show that the tracking controllability of the heat equation
is subordinated to the analogous property of the wave equation.
The tracking controllability of the wave equation has been mainly analyzed for $d=1$,  first in \cite{li2010exact} and \cite{li2016exact}, 
with constructive methods and then, in \cite{sarac2021sidewise}, \cite{zuazua}, by means of a duality approach inspiring this paper, and, finally, in \cite{Dehman2024boundary} in the multi-dimensional setting, employing microlocal analysis techniques.

A third contribution of this paper, presented in Section \ref{sec:segcyl}, concerns the quantification of the cost
of approximate controllability for the heat equation. This is done by
carefully analyzing power series representations, a method that, as mentioned above, has already been used to tackle the tracking control of the 1$d$ heat equation.

The results of this paper can be extended to other situations:  
the control may act on Neumann boundary conditions 
and aim at regulating the Dirichlet trace; the heat equation may involve variable coefficients; the model under consideration could be nonlinear, etc.
Section \ref{sec:opprob} is devoted to present some of these variants and other interesting and challenging open problems.

\section{Framework for tracking controllability}\label{sec:gendom}

In this section we formulate  the tracking controllability problem in an abstract setting,   taken from \cite[Section 2.3]{coron2007control}, to later apply it to the heat equation.
We refer the reader to \cite{sarac2021sidewise} for the corresponding wave-problem.

\subsection{An abstract setting}

Consider the abstract controlled model
\begin{equation}\label{con:abssys}
\begin{cases}
y_{t}=Ay+Bu,\\
y(0)=y_0,& 
\end{cases}
\end{equation}
the target being goal-oriented
\begin{equation}\label{eq:Ew}
Ey(t)=w(t)  \mbox{ on } (\tau,T),
\end{equation}
for some $\tau\geq0$, i. e. focusing on the projection of the state $y$ through the operator $E$. 

Here $A:D(A)\to Y$ is assumed to be the infinitesimal generator of a continuous semigroup, and $B:U\to D(A^*)'$ and $E:D(A)\to W$  bounded linear operators.
Moreover $Y$, $U$ and $W$ are Hilbert spaces endowed with the scalar products $\langle\cdot,\cdot\rangle_Y$,  $\langle\cdot,\cdot\rangle_U$ and $\langle\cdot,\cdot\rangle_W$ respectively.   

As it is classical in control problems, we consider 
the dual problem,
which reads as follows:
 \begin{equation}\label{eq:adjabstract}
 \begin{cases}
 -p_t=A^*p+E^*f,\\
 p(T)=0.
 \end{cases}
 \end{equation}
 
Based on the Hilbert Uniqueness Method (HUM),
 we can obtain the dual characterisation  of the problem of approximate sidewise or tracking controllability:

\begin{prop}[Duality for approximate  tracking control]\label{prop:absdualapprx}

 For all $w\in L^2(0,T;W)$ and $\eps>0$
there is a control $u\in L^2(0,T;U)$ such that the solution of
\eqref{con:abssys} satisfies
\begin{equation}\label{eq:approxheat}
\|Ey-w\|_{L^2(0,T;W)}<\eps,
\end{equation}
if and only if the following uniqueness or unique continuation property is satisfied: for all $f\in L^2(0,T;U)\setminus \{0\}$  
the solution $p_f$ of \eqref{eq:adjabstract} satisfies 
\begin{equation}\label{eq:pfnotnull}
B^*p_f\neq 0.
\end{equation}

When the equivalent unique continuation property above for the adjoint system holds, in the particular case where 
$y_0=0$ (which, by the linearity of the system, can be considered without loss of generality), the approximate control of minimal norm takes the form
$$v=B^*p_f, $$ where $f$ is
the minimizer of:
\begin{equation}\label{eq:defJf}
\begin{split}
J(f)=\frac{1}{2}\|B^*p_f\|^2_{L^2(0,T;U)}
&- \int_0^T\langle f,w \rangle_W dt 
\\&+\eps \|f\|_{L^2(0,T;W)}.
\end{split}
\end{equation}
\end{prop}

\medskip

\begin{proof}
Proposition \ref{prop:absdualapprx} is standard in the context 
of HUM (see  \cite{lions1988controlabilite}, and  \cite[Section 2.3]{coron2007control}).

As observed above, by the linearity of the system, it suffices to prove the approximate controllability
for $y_0=0$. 

Let us suppose that the unique continuation property holds for the adjoint system, i.e. \eqref{eq:pfnotnull} is satisfied for all
$f\in L^2(0,T;U)\setminus\{0\}$.
Then, $J$ is strictly convex, continuous and coercive, and it has
a unique  minimizer  $\tilde f\in L^2(0,T;U)$.

The Euler-Lagrange equations assure that, for  all $f\in L^2(0,T;U)$
and $\delta\neq0$,
\begin{equation}\label{eq:defpf}
\begin{split}
&\delta\int_{0}^T\langle B^*p_{\tilde f},B^* p_f\rangle_U dt
-\delta \int_{0}^T\langle f,w\rangle_Wdt
\\&+\eps(\|\tilde f+\delta f\|_{L^2(0,T;W)}-\|\tilde f\|_{L^2(0,T;W)})+O_{\delta\to0}(\delta^2)
\\&=J(\tilde f+\delta f)-J(\tilde f)\geq0.
\end{split}
\end{equation}
Moreover, if  $y$ is the solution of \eqref{con:abssys} with $y_0=0$
and $v=B^*p_{\tilde f}$, then:
\begin{equation}\label{eq:dualcompint}
\begin{split}
0&=\int_0^T\langle y_t-Ay-BB^*p_{\tilde f}, p_f\rangle_Y dt
\\&=\int_0^T \langle y,E^*f\rangle_Y
-\int_0^T \langle BB^*p_{\tilde f},p_f\rangle_Y dt,
\end{split}
\end{equation}
which implies that:
\begin{equation}\label{eq:bdeq}
\int_0^T\langle B^*p_{\tilde f},B^* p_f\rangle_U dt=
\int_0^T\langle  Ey,f\rangle_W dt.
\end{equation}
Thus, combining \eqref{eq:defpf}-\eqref{eq:bdeq}, 
the solution of \eqref{con:abssys} with control $v=B^*p_{\tilde f}$ satisfies:
\begin{equation}\label{eq:deltafunceps}
\begin{split}
&\delta\int_0^T\langle  Ey-w,f\rangle_W dt+O(\delta^2)
\\&\geq -\eps(\|\tilde f+\delta f\|_{L^2(0,T;W)}-\|\tilde f\|_{L^2(0,T;W)})
 \\&\geq -\eps|\delta| \|f\|_{L^2(0,T;W)}.
\end{split}
\end{equation}
Taking $\delta\to 0^+$ and $\delta\to0^-$, we obtain from \eqref{eq:deltafunceps} that:
\[\left|\int_0^T\langle  Ey-w,f\rangle_Wdt \right|
\leq\eps \|f\|_{L^2(0,T;W)},
\]
for all $f\in L^2(0,T;W)$, which implies
\eqref{eq:approxheat}.

Reciprocally, if $B^*p_f=0$ for some $f\neq0$, 
considering \eqref{eq:dualcompint}, $E y$ is orthogonal
to $f$ for all $v\in L^2(0,T;U)$, and
the system \eqref{con:abssys} is not approximately controllable.
\end{proof}

In a similar way, based on  HUM,
 we can  also obtain the dual characterization for exact sidewise or tracking controllability property:
\begin{prop}[Duality for exact tracking
controllability]\label{prop:dualapprx}

{ \em For all $w\in L^2(0,T;W)$ there is a control $f\in L^2(0,T;U)\setminus \{0\}$ such that the solution of
\eqref{con:abssys} satisfies:
\begin{equation}\label{eq:exactheat}
Ey=w,
\end{equation}
if and only if 
\begin{equation}\label{eq:ratioabs}
\sup_{f\in L^2(0,T;U)\setminus \{0\}} \frac{\|f\|_{L^2(0,T;W)}}
{\|B^*p_f\|_{L^2(0,T;U)}}<+\infty,
\end{equation}
for all solutions $p_f$ of \eqref{eq:adjabstract}.

When \eqref{eq:ratioabs} is satisfied,  the control of minimal norm (with $y_0=0$) is given by
$v=B^* p_f $, where $f$ is
the minimizer of:
\begin{equation}\label{eq:defJfexact}
J(f)=\frac{1}{2}\|B^*p_f\|^2_{L^2(0,T;U)}
- \int_0^T\langle f,w \rangle_W dt.
\end{equation}
}
\end{prop}

\begin{rem}\label{rk:obsineg}
Despite of the abundant existing literature on Carleman inequalities to prove observability inequalities for heat-like equations, the authors are not aware of any results allowing to obtain inequalities of the form \eqref{eq:ratioabs}. In the context of the wave equation, this issue has been successfully addressed in the papers mentioned above, in which such inequalities have been derived using sidewise energy estimates and microlocal analysis tools.
\end{rem}

\subsection{Tracking control for the heat equation}\label{sec:obsprob}

The following result on the  sidewise or tracking  
approximate controllability of the heat equation is a consequence of  Proposition \ref{prop:absdualapprx}: 

\begin{prop}[Approximate tracking  control]\label{cl:trckcon}
{\em  Let $\Omega$ be a $C^1$ domain, $\gamma\subset
\partial\Omega$ be relatively open and non-empty, and 
$\tilde\gamma\subset\subset\partial\Omega\setminus\gamma$. 
Then, for all $w\in L^2((0,T)\times\tilde\gamma)$ and $\eps>0$
there is a control $v\in L^2((0,T)\times \gamma)$ such that the solution of
\eqref{con:heat} satisfies: 
\[\|\partial_\nu y-w\|_{L^2((0,T)\times\tilde\gamma)}\le \eps.\]
}
\end{prop}
\medskip

\begin{proof}
The dual system of \eqref{con:heat}-\eqref{eq:parny}
reads as follows:
\begin{equation}\label{adj:heat}
\begin{cases}
-p_{t}-\Delta p=0 & \mbox{ in } (0,T)\times\Omega,\\
p=f1_{\tilde\gamma} & \mbox{ on } (0,T)\times\partial\Omega,\\
p(T)=0 & \mbox{ in }\Omega.
\end{cases}
\end{equation}
By Proposition
\ref{prop:dualapprx}, it suffices to prove that $\partial_\nu p_f=0$
on $(0,T)\times\gamma$ implies that $f=0$.
This is a consequence of  Holmgren's Uniqueness Theorem  (see, for instance, \cite[Theorem 8.6.5]{hormander2007analysis}), as can be easily checked by a classical argument. Indeed, if $\partial_\nu p=0$ on $(0,T)\times\gamma$, given that $p=0$ on $(0,T)\times\gamma$, $p$ together with all the first order derivatives of $p$ vanish on $\gamma$. Thus, we can extend the solution to $0$ on a small neighbourhood of $\gamma$, to an extended solution in an enlarged domain, vanishing on an open set. Thanks to Holmgren's Theorem, that assures the well known unique continuation property of solutions of the heat equation in an arbitrarily small time interval, we conclude that  $f=0$.
\end{proof}

\begin{rem}[Regularity of the trackable space]\label{rk:traksp}
As explained in \cite{fabre1995approximate}, it is not straightforward that 
$\partial_\nu y$ belongs to $L^2((0,T)\times\tilde\gamma)$. This can be proved by considering the solution of the heat equation as transposition.
In fact, due to the regularizing effect of the heat equation, we cannot 
expect that the trackable space contains irregular traces if 
$\tilde\gamma\subset\subset\partial\Omega\setminus\gamma$.
By means of  classical bootstrap arguments, as in \cite[Lemma 2.5]{doubova2002controllability}, it can be shown that the reachable space must be constituted by regular functions (notably, if $\Omega$ is a $C^\infty$ domain, the trace must be $C^\infty$). One should expect reachable targets  to be of Gevrey regularity, but determining the sharp space is an interesting open problem.
\end{rem}

\section{Transmutation for tracking controllability}\label{sec:trasm}

In this section we relate the tracking controllability properties of the heat  and the wave equations by using a variant of the
 Kannai transform (see \cite{kannai1977off}, \cite{miller2004geometric} and \cite{miller2006control}), 
which consists, roughly,  on averaging the solutions of the wave equation with
the heat kernel
 \begin{equation}\label{def:k}
k(t,s):= \frac{e^{-s^2/(4t)}}{\sqrt{4\pi t}}, 
\end{equation}
i.e.  the fundamental solution of the heat equation:
\begin{equation}\label{eq:kernel}
\partial_t k=\partial_{ss} k;\ \ \  k(0,s)=\delta_0(s).
\end{equation}

To be more precise, let us consider  the following control problem for the wave equation:
\begin{equation}\label{eq:trackwaved}
\begin{cases}
z_{ss}-\Delta z=0 & \mbox{ in }\bb R\times \Omega,\\
z(s,\cdot)= g1_\gamma & \mbox{ on }  \bb R\times \partial\Omega,\\
z(0,\cdot)=z_0& \mbox{ in }\Omega,\\
z_s(0,\cdot)=z_1& \mbox{ in }\Omega.
\end{cases}
\end{equation}
Note that in this wave equation the (pseudo-)time variable is denoted by $s \in \bb R$, to distinguish it from the real time-variable, $t$, along which the heat process evolves.

Here, $\Omega\subset\bb R^d$ is a $C^2$ domain, 
$g$ is the $L^2$-control and 
$(z_0,z_1)\in L^2(\Omega)\times H^{-1}(\Omega)$ the initial states. 

We define the functional space:
\begin{equation*}
\begin{split}
\cal E(\bb R; H):= &\left\{g\in L^\infty_{\op{loc}}(\bb R;H):
\forall\delta>0\ \exists C_\delta>0: \right. \\& \left. 
\|g(t)\|_{H}\leq C_\delta e^{\delta t^2} \ \forall t\in\bb R\right\},
\end{split}
\end{equation*} 
for a given Hilbert space $H$.

The adaptation of the Kannai or transmutation transform to this setting reads as follows:
\begin{prop}[Kannai transform]\label{lm:Kannaitrans}
{\em Let $\Omega$ be a $C^2$ domain, $\gamma\subset\partial\Omega$,
 $g\in\cal E(\bb R, L^2(\gamma))$, $z_0\in L^2(\Omega)$, $z_1=0$ and $z$ be the corresponding solution of 
\eqref{eq:trackwaved}. Then, 
\begin{equation}\label{eq:ytransf}
y(t,x)=\int_{-\infty}^{\infty}k(t,s)z(s,x)ds,
\end{equation}
is a solution of \eqref{con:heat} for $T=\infty$, $y_0=z_0$,
$$v(t,x):=\int_{-\infty}^{\infty}k(t,s)g(s,x)ds,$$
and it satisfies:
$$\partial_\nu y(t,x)=\int_{-\infty}^{\infty}k(t,s)\partial_\nu z(s,x)ds.$$
}
\end{prop}

\begin{proof}
First, it is easy to see  that the function $y$ given by \eqref{eq:ytransf}
satisfies the boundary conditions of \eqref{con:heat}. 

Moreover, it
satisfies the initial condition because of \eqref{eq:kernel}. 

Finally, it is a solution of the heat equation. Indeed, if $g\in\cal D(\bb R\setminus\{0\}\times\gamma)$,
then for all $t\in(0,T)$ and $x\in(0,L)$ the following  holds:
\begin{equation*}
\begin{split}
y_t&=\int_{-\infty}^{\infty}k_t(t,s)z(s,x)ds
=\int_{-\infty}^{\infty}k_{ss}(t,s)z(s,x)ds
\\&=\int_{-\infty}^{\infty}k(t,s)z_{ss}(s,x)ds
=\int_{-\infty}^{\infty}k(t,s)\Delta z(s,x)ds
\\&=\Delta\left(\int_{-\infty}^{\infty} k(t,s)z(s,x)ds\right)
=\Delta y.
\end{split}
\end{equation*}
We have used \eqref{eq:kernel} in the second equality. Note also that
the integration by parts on the third equality is rigorous because 
$k$ decays exponentially when $s\to\infty$ and $v(s)$ grows at most linearly.

Finally, by density,  it follows that $$\int_{-\infty}^{\infty}k_{ss}(t,s)z(s,x)ds= \Delta\int_{-\infty}^{\infty} k(t,s)z(s,x)ds$$  for any $g\in \cal E(\bb R;L^2(\gamma))$, since, for all $t>0$, the function $e^{-s^2/(4t)}z$ decays quadratic exponentially when $s\to\infty$, so 
$y_t=\Delta y$.
\end{proof}

\begin{rem}
This transmutation identity allows to transfer the tracking controllability properties from the wave to the heat equation. In particular, in the 1$d$ setting, it  allows to achieve precise results, in combination with those in \cite{zuazua}. 

Indeed, if the control $g$ assures tracking the trace $h$ for the wave equation, then, the control 
$$
v(t,x)=\int_{-\infty}^{\infty}k(t,s)g(s,x)ds,
$$
allows to track the trace 
$$
w(t,x)=\int_{-\infty}^{\infty}k(t,s)h(s,x)ds,
$$
for the heat equation.
\end{rem}

\section{Tracking control of the 1$d$ heat equation}
 \label{sec:segcyl}

In this section we study the tracking controllability of
the 1$d$ heat equation by using the flatness approach. 
Notably, we study the solutions of:
\begin{equation}\label{con:heat0}
\begin{cases}
y_{t}-\partial_{xx} y=0 & \mbox{ in } (0,T)\times(0,L),\\
y(\cdot,0)=0 &\mbox{ on } (0,T),\\
y(\cdot,L)=v& \mbox{ on } (0,T),\\
y(0)=0 & \mbox{ in }(0,L).
\end{cases}
\end{equation}

First, we recall
how to compute explicitly the controls so that the solution of \eqref{con:heat0}
satisfies: 
\begin{equation} \label{tarjet:heat0}
\partial_x y(\cdot,0)=w  \ \ \ \  \mbox{ on } (0,T),
\end{equation}
where $w$ is a flat function. Then,
we use those controls to get an upper bound on the cost for approximate 
tracking controllability. In particular, we derive   
an upper bound for the norm of the control which, acting on $(0,T)\times\{L\}$,  assures that
\begin{equation}\label{tarjet:heatapprox0}
\|\partial_x y(\cdot,0)-w\|_{C^0([0,T])}\leq\eps.
\end{equation}

We prove the following result:

\begin{thm}[Cost of approximate tracking control]\label{tm:approxcost}
{ \em Let $L>0$, $T>0$ and $w\in W^{1,\infty}(0,T)$ be a 
function satisfying $w(0)=0$. Then,
for all $s\in (0,1)$ there is a constant $C=C(s)>0$ such that
 for all $\eps>0$ there exists a control $v_\eps$
satisfying
\begin{equation}\label{est:costveps}
\begin{split}
&\|v_\eps\|_{C^0([0,T])}\\&\leq C \exp\left[
C\left(\frac{\|w\|_{W^{1,\infty}(0,T)}}
{\eps}\right)^{1/s}\right]\|w\|_{W^{1,\infty}(0,T)},
\end{split}
\end{equation}
and such that the solution of \eqref{con:heat0} satisfies
\eqref{tarjet:heatapprox0}.}
\end{thm}\noindent
The proof of Theorem \ref{tm:approxcost} consists on approximating 
the target with Gevrey functions in the $C^0$-norm, obtained by  convolution with cut-off functions, and estimating the controls for those approximating targets.

The requirement
of $w$ being  more regular than the approximation space is standard in approximation theory in infinite-dimensional spaces.

Before proving Theorem \ref{tm:approxcost}, we recall some
results about Gevrey functions, and
prove some technical estimates on an auxiliary cut-off function that will be employed in our discussion.
 
In this section $C>0$ is a generic constant that changes from line to line.

\subsection{Preliminaries}  

In this section we recall the known controllability
results for flat functions in 1$d$ and on  approximation rates by means of Gevrey cut-off functions. 

By definition $w$ is a Gevrey function 
of order $r$ if and only if 
$w\in C^\infty([0,T])$ and it satisfies for some $C,R>0$:
\[|w^{(i)}(t)|\leq C\frac{(i!)^r}{R^i}, \ \ \forall t\in[0,T], \forall i\geq0.
\]
When $r=1$ Gevrey functions are analytic.

\begin{lem}[Controls for flat targets]\label{lm:trackGevrey}
{\em Let $r\in [1,2)$, $L>0$, $T>0$
 and $w\in C^\infty([0,T])$ be a Gevrey function of order 
$r$ satisfying $w^{(i)}(0)=0$ for all $i\in\bb N$. Then,
there is a control $v$,  a Gevrey function of order $r$, such that
the solution of \eqref{con:heat0} satisfies 
\eqref{tarjet:heat0}.}
\end{lem}\noindent

The proof of Lemma \ref{lm:trackGevrey} is mainly contained in 
\cite{laroche2000motion}. The procedure in \cite{laroche2000motion} consists on considering controls of the form  \begin{equation*}
v(t)= \sum_{i\geq0}\frac{L^{2i+1}}{(2i+1)!}w^{(i)}(t),
\end{equation*}
and the corresponding solutions of \eqref{con:heat0},  given by:
\begin{equation}\label{def:solyline}
y(t,x)=\sum_{i\geq0}\frac{x^{2i+1}}{(2i+1)!}w^{(i)}(t).
\end{equation}

Next, we recall the existence of 
cut-off  functions in Gevrey spaces:
\begin{lem}[Gevrey cut-off functions]\label{lm:defxi}
{\em Let $r>1$. There is a cut-off function $\xi$ supported
in $[0,1]$ of Gevrey order $r$ and satisfying $\int_0^1\xi(t)dt=1$.}
\end{lem}\noindent
In fact, the following function is of order $r$: $$\exp\left( \frac{-1}{((1-t)t)^{1/(r-1)} }\right)1_{(0,1)}.$$  
 This was  first proved in \cite{ramis1978devissage} using Cauchy's integral and Stirling's formula (see
\cite[Lemma 4]{schorkhuber2013flatness} for an English version). 

\subsection{Upper bounds for special auxiliary functions}
 In order to quantify the
cost we introduce the special auxiliary functions:
\begin{equation}\label{def:Gsx}
\cal G_s (x):= \sum_{i\geq 0}\frac{x^i}{(i!)^s}.
\end{equation}

These functions have an exponential growth:
\begin{lem}[Upper bounds for $\mathcal G_s$] \label{lm:bdexp}
{\em Let $s\in(0,1)$. Then, there is $C>0$ depending on $s$ such that:
\begin{equation}\label{est:bdGtheta}
\cal G_s(x)\leq C\exp\left(Cx^\frac{1}{s}\right), \ \ \ \forall x\geq0. 
\end{equation}}
\end{lem}\noindent

\begin{proof}
The proof consists on estimating $\partial_x[\ln(\cal G_s(x))]$
with Stirling's formula and  splitting the lower and  higher order terms 
of the sum.

In order to prove \eqref{est:bdGtheta} it suffices to prove that
there is $C>0$ such that:
\begin{equation}\label{est:Gthetaprim}
\cal G_s'(x)\leq Cx^{\frac{1-s}{s}}\cal G_s(x),  \ \ \ 
\forall x\geq 1.
\end{equation}
For that purpose, we remark that:
\[\cal G_s '(x)=\sum_{i\geq1}i^{1-s}\frac{x^{i-1}}{[(i-1)!]^s }
=\sum_{i\geq0}(i+1)^{1-s }\frac{x^i}{(i!)^s }.
\]
In order to prove \eqref{est:Gthetaprim} we split  
the terms into $i< 2x^{1/s}e$ and
$i\geq 2x^{1/s}e$. On the one hand,
if $i\geq 2x^{1/s}e$, from Stirling's formula we get that:
\begin{equation}\label{est:xiitheta}
\frac{x^i}{(i!)^s}\leq C \frac{x^ie^{is}}{i^{is}}
=C\left(\frac{x^{1/s}e}{i}\right)^{is}\leq C2^{-is}.
\end{equation}
Thus, from \eqref{est:xiitheta} we obtain that,  for all $x\geq1$,
\begin{equation}\label{est:sumupterm}
\sum_{i\geq 2x^{1/s}e}(i+1)^{1-s}\frac{x^i}{(i!)^s}
\leq C\leq Cx^{\frac{1-s}{s}}\cal G_s(x).
\end{equation}
On the other hand,  for all $x\geq1$ it holds
\begin{equation}\label{est:sumdownterms}
\begin{split}
\sum_{i< 2x^{1/s}e}(i+1)^{1-s}\frac{x^i}{(i!)^s}
&\leq  \sum_{i<2x^{1/s}e}(4e)^{1-s} x^{\frac{1-s}{s}}\frac{x^i}{(i!)^s}\\&\leq C x^{\frac{1-s}{s}}\cal G_s(x).
\end{split}
\end{equation}
Therefore, \eqref{est:Gthetaprim} follows from \eqref{est:sumupterm} and 
\eqref{est:sumdownterms}.
\end{proof} 
 
\begin{rem}\label{rem:sgeq1lem}
\eqref{est:bdGtheta} is also true if $s\geq1$, though we
do not use it in this paper. This can be proved as follows:
\[\begin{split}
\sum_{i\geq0}\frac{x^i}{(i!)^s}&=
\sum_{i\geq0}\left(\frac{(x^{1/s})^i}{i!}\right)^s\\&\leq
\left(\sum_{i\geq0}\frac{(x^{1/s})^i}{i!}\right)^s=
\exp(sx^{1/s}),
\end{split}
\]
using that, if $s\geq 1$, $k\in\mathbb N$ 
and $a_1,\ldots,a_k\geq0$, then 
$$(a_1^s+\cdots+a_k^s)\leq (a_1+\cdots+a_k)^s.$$
\end{rem}

\begin{rem}
When $s\in(0,1)$, in a similar way as
Remark \ref{rem:sgeq1lem}, we may show that:
\[\mathcal G_s(x)\geq \exp(sx^{1/s}).
\]
Thus, combining this with Lemma \ref{lm:bdexp} and using \cite[Proposition 4.3]{martin2016reachable}, we obtain that $\mathcal G_s$ are entire functions of order $1/(1-s)$. 
\end{rem}

\subsection{Conclusion of the proof of Theorem \ref{tm:approxcost}}

We now have all the ingredients to prove the upper bound of the cost
of approximate controllability. The proof of Theorem \ref{tm:approxcost}
is divided in two steps: first, we approximate the target $w$
by convolution with the cut-off function given in Lemma \ref{lm:defxi},
and, secondly, we apply the control given in Lemma \ref{lm:trackGevrey} 
and estimate  it with Lemma \ref{lm:bdexp}.

\begin{proof} Recall that  $C>0$ is a generic constant changing from line to line.

\textit{Step 1: Approximation of the target.}
Let $s\in (0,1)$ and $w\in W^{1,\infty}(0,T)$ a function
satisfying $w(0)=0$. 
 Define $\xi_\delta:= \delta^{-1}\xi(x\delta^{-1})$,
where $\xi$ is the Gevrey function of order $r=2-s$
given in Lemma \ref{lm:defxi}. Set
\begin{equation}\label{def:weps}
w_\delta:= \tilde w\ast\xi_\delta=
\int_{t-\delta}^t\tilde w(t')\xi_\delta(t-t')dt',
\end{equation}
where $\tilde w$ is the extension
of $w$ by $0$ to $\bb R^-$.
Since $\xi$ is supported in $[0,1]$ and since $w_\delta=0$
in $(-\infty,0]$,  $w_\delta$ vanishes at 
$t=0$. Moreover, from $w(0)=0$ we get that:
\begin{equation}\label{eq:wdelta}
\begin{split}
|w_\delta(t) -w(t)|& \leq \sup_{t'\in (0,\delta)}|\tilde w(t-t')-w(t)|
\\&\leq \delta \|w\|_{W^{1,\infty}(0,T)} \ \ \forall t\in[0,T].
\end{split}
\end{equation}
Thus, from now on we consider:
\begin{equation}\label{eq:deltaeps}
\delta:=\frac{\eps}{\|w\|_{W^{1,\infty}(0,T)}},
\end{equation}
so \eqref{eq:wdelta} turns into:
\begin{equation}\label{eq:approxwC0}
\|w_\delta-w\|_{C^0([0,T])}\leq \eps . 
\end{equation}
Since $\xi$
 is a Gevrey function of order $r=2-s$, $w_\delta$ is easily seen to be a Gevrey function of order $r=2-s$ as well.
In fact, considering  \eqref{def:weps}
 we get that:
\begin{equation}\label{est:wepsN}
\|w_\delta\|_{C^i([0,T])}\leq \delta^{-i}\|\xi\|_{C^i([0,1])}
\|w\|_{W^{1,\infty}(0,T)},
\ \ \ \forall i\in\bb N.
\end{equation}
Thus, from the assumption that $\xi$ is a Gevrey function of
order $r=2-s$ and \eqref{est:wepsN} we deduce that:
\begin{equation}\label{est:wepsNcomplete}
\|w_\delta\|_{C^i([0,T])}\leq \left(\frac{C}{\delta}\right)^{i}(i!)^{2-s},
\ \ \ \forall i\in\bb N.
\end{equation}

\textit{Step 2: Cost of tracking control.}  
From Lemma \ref{lm:trackGevrey} we obtain that
 $\partial_xy(\cdot,0)=w_\delta$ with the control
 \begin{equation*}
v_\delta(t)= \sum_{i\geq0}\frac{L^{2i+1}}{(2i+1)!}w_\delta^{(i)}(t).
\end{equation*}
In particular, from \eqref{est:wepsNcomplete} we find that:
\begin{equation}\label{def:flatlinev}
\|v_\delta\|_{C^0([0,T])}\leq
\sum_{i\geq0}\left( \frac{C}{\delta}\right)^{i}
\frac{(i!)^{2-s}}{(2i+1)!}\|w\|_{W^{1,\infty}(0,T)} ,
\end{equation}
for $C$ a constant independent of $i$. Next, we consider
that:
\begin{equation}\label{est:facti}
\frac{(i!)^{2-s}}{(2i+1)!}\leq \frac{1}{(i!)^{s}},
\end{equation}
since:
\[\frac{(2i)!}{(i!)^2}={2i \choose i}>1.
\]
Thus, from  \eqref{eq:deltaeps}, \eqref{def:flatlinev}, and \eqref{est:facti}: 
\begin{equation}\label{eq:estcostvdelta}
\begin{split}
\|v_\delta\|_{C^0([0,T])}&\leq
\sum_{i\geq0}\left( \frac{C}{\delta}\right)^{i}\frac{(i!)^{2-s}}{(2i+1)!}\|w\|_{W^{1,\infty}(0,T)}
\\&\leq \dis\sum_{i\geq0}
\left(\frac{C}{\delta}\right)^{i}\frac{1}{(i!)^{s}}\|w\|_{W^{1,\infty}(0,T)}
\\&=\cal G_s\left(C\frac{\|w\|_{W^{1,\infty}(0,T)}}{\eps}\right)\|w\|_{W^{1,\infty}(0,T)}.
\end{split}
\end{equation}
Hence, we obtain \eqref{est:costveps} from \eqref{est:bdGtheta} and \eqref{eq:estcostvdelta}.
\end{proof}

  \section{Open problems}\label{sec:opprob}

As mentioned above, our techniques apply also for other boundary conditions.

  The method and results in this paper lead to some interesting
open problems and could be extended in various directions
(in addition to the ones proposed in Remarks \ref{rk:obsineg} and \ref{rk:traksp}). Namely:

\begin{itemize}
\item \textbf{Multi-dimensional domains.} 
Getting more precise quantitative results for the tracking control of the multi-dimensional heat equation is an interesting open problem. The combination of the results  in \cite{Dehman2024boundary} on the multi-dimensional wave equation and the transmutation formula above is a promising path. The results presented in \cite{strohmaier2022analytic}, which generalize those in \cite{darde2018reachable}, demonstrate that the reachable space is bounded between two spaces of holomorphic functions. These findings deserve also mention.

\item \textbf{Optimality on the cost of approximate controllability.}
One relevant open problem  is whether 
 the  upper bounds  in Theorem 
\ref{tm:approxcost} can be sharpened  to obtain $\exp(C\eps^{-1})$ or $\exp(C\eps^{-1/2})$,
 in line with the known bounds for the classical
approximate controllability problem of parabolic equations. 
Indeed,
the cost of driving the heat equation dynamics to a $L^2$-distance of the order of  $\eps$  to a target  function $y^T\in H^2(0,L)\cap H_0^1(0,L)$
is bounded  above by $\exp(C\eps^{-1/2})$ for the heat equation with constant
coefficients  (see \cite{fernandez2000cost}).
An estimate  of  the order $\exp(C\eps^{-1})$
 holds as well for more general heat equations
(see \cite{fernandez2000cost},
\cite{phung2004note} and \cite{boutaayamou2020cost}),
for the semi-linear heat equation (see \cite{yan2009cost}), for
the Ginzburg-Landau equation (see \cite{aramua2017cost}) and 
for the hypoelliptic heat equation (see \cite{laurent2017tunneling}). 
Their proofs are based on  observability inequalities
obtained through Carleman inequalities, with appropriate weight functions. This is an open issue in the context of sidewise or tracking controllability. 
\item \textbf{Sidewise observability estimates for the heat equation.} The
obtention of sidewise observability inequalities for system 
\eqref{adj:heat} remains open. This is closely related with the problem above of getting sharp bounds for approximate sidewise controllability. Whether Carleman inequalities can be adapted in this setting is not yet well understood.
It would also be interesting to describe whether the results in  Lemma 
\ref{lm:trackGevrey} in Gevrey spaces can be related to some sidewise observability inequality, \cite{dasgupta2014gevrey}. 
  
\item \textbf{Simultaneous tracking  and null control.}
The problem of finding controls that simultaneously drive the state to rest and assure the tracking control property is an interesting open problem. This problem has been successfully addressed in \cite{zuazua} for the 1$d$ wave equation.

\item \textbf{Other parabolic models.} 
It would also be interesting to analyse these problems for systems of parabolic equations and the Stokes system, for instance. One could also consider more general systems, as for instance, thermoelasticity, merging the behaviour of the wave and heat equations. But more systematic methods for sidewise and tracking controllability should be developed to be in conditions to tackle these problems.

 \end{itemize}

\section{Conclusions}\label{sec:concl}
In this paper we have studied the tracking controllability for the
heat equation and its relation with the sidewise controllability of
the wave equation.
We have shown that duality methods and classical results on unique continuation allow to prove approximate tracking
controllability properties for the heat equation in all space dimensions. 
We have also shown how  transmutation methods can be adapted to this setting, achieving the tracking control of the heat equation, out of the corresponding properties of the wave equation.

Revisiting the flatness approach we have also obtained estimates on the cost 
 of approximate tracking controllability for the 1$d$ heat equation.

In the future, efforts should be devoted to develop more systematic methods to tackle these problems and, in particular, the open questions mentioned in the previous section, among others.

\section*{References}

\bibliographystyle{plain}        
\bibliography{Nodalheat}

\end{document}